\documentclass[a4paper]{article}
\usepackage{amsmath}
\usepackage{amsthm} 
\usepackage{amssymb}
\usepackage{amsfonts}
\usepackage{amssymb}
\usepackage{graphicx, tikz}
\usepackage[colorlinks=true, urlcolor=blue, linkcolor=red]{hyperref}
\newtheorem{definition}{Definition}
\newtheorem{thm}{Theorem}
\newtheorem{cor}{Corollary}
\newtheorem{lemma}{Lemma}
\newtheorem{prop}{Proposition}
\newcommand{\bq}{\mathbf{Q}}
\newcommand{\bz}{\mathbf{Z}}

\newcommand{\bl}{\begin{lemma}}
\newcommand{\el}{\end{lemma}}
\newcommand{\bd}{\begin{definition}}
\newcommand{\ed}{\end{definition}}
\newcommand{\bcor}{\begin{cor}}
\newcommand{\ecor}{\end{cor}}
\newcommand{\gal}{\mathrm{Gal\,}}
\title{An Elementary Problem in Galois Theory 
about  the Roots of  Irreducible Polynomials}
\author{M Krithika and P Vanchinathan*}
\begin{document}
\maketitle
\begin{center}
	School of Advanced Sciences\\
	VIT  Chennai\\
	Vandalur-Kelambakkam Raod\\
	Chennai 6000127, INDIA\\
 \texttt{krithika.m2020@vitstudent.ac.in, vanchinathan.p@vit.ac.in}
\end{center}

\begin{abstract}
	For a  field  $K$, and a root $\alpha$ of an irreducible polynomial over $K$ (in some algebraic closure) the number of roots of $f(x)$ lying in $K(\alpha)$ is studied here. Given such an $f(x)$ of degree $n$ for which $r$ of the roots are in $K(\alpha)$, we describe a  construction that yields, for $d\ge2$, irreducible polynomial $g(x)$ of degree $nd$ 
and with  exactly $rd$ of the roots in the field generated by any one root of those polynomials. Our results are valid for fields of characteristic 0 
and possibly all infinite  perfect fields.
As an application, for any algebraic number field $K$ and   positive integers $n\ge3,d\ge2$,
 we provide  irreducible polynomials  of degree $nd$
with exactly  $d$ roots in the field generated by one of the roots.

Independently,	for $k<n$, we prove directly the existence
of irreducible polynomials over $\bq$  of  degree $n!/(n-k)!$
for which the field generated by one root contains exactly
	$k!$ roots.

Many interesting  new questions for further research are provided. 
\end{abstract}

\noindent\textbf{Keywords:} Roots of polynomials; root cluster; Galois theory

\vspace{3pt}
\noindent\textbf{AMS Classification:}  11R32, 12F05
\section{Introduction}

In Galois theory the splitting field of an irreducible polynomial $f(x)$, 
i.e., the field obtained by adjoining \textit{all the roots of} $f(x)$,
is the main object of study. The very definition of this
implicitly admits the reality that, with just a single  root, the field 
generated may be  much smaller than the splitting field. As an extreme 
case, given an   odd number $n$, and a prime number $p$, for the 
irreducible polynomial $x^n-p \in\bq[x]$, the field field generated 
by the unique real root (in fact, any other root) will contain no other roots.

So the following question arises naturally: for a general irreducible polynomial  
how many of the roots are found in  the field generated by a single  root?
Alexandar Perlis in his paper \cite{perlis} published in 2004   has given a clear interpretation for this number in terms of the subgroups of the Galois group. (See Theorem~\ref{perthm} in Section 2). In this context
we introduce the following  terminology (deviating from his own):

\bd
For an irreducible polynomial $f(x)$ over some perfect field $K$,
and  a root $\alpha$ of $f(x)$ in a fixed algebraic closure $\bar K$, 
by the  root cluster  of $\alpha$  we mean the set of roots of $f(x)$ in the field $K(\alpha)$. The cardinality of this set is called the root cluster size
of $\alpha$. 
\ed
For simplicity sometimes we say cluster size to mean root cluster size.

\vspace{3pt}

We have already seen the extreme case where the clusters are singletons.
The other extreme is obvious:

\vspace{3pt}
\noindent\textbf{Example 1}\quad
 For  any primitive element $\alpha$  of a Galois extension $L/K$,
all the roots of the minimal polynomial of $\alpha$ over $K$ form a 
single cluster. So the  root cluster size is maximal, namely the degree
of the irreducible polynomial. 

Next we give an example that is in the  middle ground.

\vspace{3pt}
\noindent\textbf{Example 2}\quad 
Take a quartic polynomial over $\bq$ with $D_4$ as Galois group, to be 
specific we may take
$x^4-tx^3-x^2+tx+1$ with  $t$ an integer${}\ge4$, 
where for any root $\alpha$, we have 
$-1/\alpha$ is also a root. The four roots of this irreducible
quartic polynomial form two clusters each of size 2.

These examples make it clear that determining the cluster sizes of
an irreducible polynomial and comparing it  to its  degree is  an interesting
and fundamental problem demanding deeper study.
We continue the work initiated by Perlis.

All the fields considered in this paper are finite extensions
of a  perfect base field $K$ contained in some fixed algebraic closure 
of it.  (Finite fields are uninteresting here as all their finite
extensions are Galois;  fields $\mathbf{R, C}$
 are also not interesting cases for
 this problem as they do not have many algebraic extensions.
 In fact one may conveniently assume  our base field to be
 an arbitrary number field).

 In this paper we prove three results on the concept of root clusters.
 The second is a generalization of the result of Perlis.

First result gives   a constructive procedure 
that starts from an \mbox{irreducible}  polynomial of some cluster size as input, 
and yields a higher degree irreducible polynomial  with proportionately 
higher cluster size as the output. More precisely:
\begin{thm}(\textbf{Cluster Magnification}): \label{algorithm}
Let $f(x)\in K[x] $ be the minimal polynomial of an algebraic  element 
$\alpha$ of degree $n>2$ over $K$ with cluster size $r$ and let 
$ K_f$ be its splitting field.
Assume that there is  a Galois extension of $K$, say of degree $d$,
which is linearly disjoint with $K_f$  over $K$.
Then  there exists an irreducible polynomial $g(x)\in K[x]$ of degree $nd$ 
with cluster size $rd$.
\end{thm}

Applying the above, along with some elementary results in algebraic number 
theory, we are able to arrive at our second result, in the cases
where the base field $K$ is any algebraic number field:
\begin{thm}  \label{application}
(a) For an algebriac number field $K$, and 
	integers  $n\geq 3, d\ge1$  there exist  irreducible polynomials 
	$f(x)\in K[x]$ of degree $nd$  with root clusters of size $d$.

(b) When  $d$ is an even number, and $K=\bq$,  there exist irreducible polynomials 
in $\bq[x]$  of degree $2d$ with root cluster size $d$.
\end{thm}
The second theorem is a generalization of an unpublished  result of
Perlis from  $\bq$ to all number fields. 
That unpublished  note is available in the  
\href{https://www.math.lsu.edu/~aperlis/publications/rootsinquanta/}{home page of Perlis.}\footnote{ 
	\url{https://www.math.lsu.edu/~aperlis/publications/rootsinquanta/} }
He had proved it  using the  result of Shafarevich on the existence of Galois extensions of $\bq$ 
having a given finite solvable group as its Galois group.
But our proof of Theorem \ref{application} is simpler and constructive.

Our third result is independent of these  and explicitly constructs, for $k< n$, polynomials of degree $n!/(n-k)!$ where the cluster sizes are all $k!$. More precisely:
\begin{thm}\label{fixkroots} Let $K$ be any perfect field admitting 
	an irreducible polynomial $f(x)$  of degree $n$ in $K[x]$ with Galois
group $S_n$. For any  $k<n$, let $L_k$ be an extension of $K$ obtained by adjoining any  $k$ roots of $f(x)$. Let $g(x)$ be the minimal polynomial over $K$
for a primitive element of $L_k$. This polynomial $g(x)$ 
of degree $n!/(n-k)!$  has clusters of
size $k!$.
\end{thm}	

Our methods are elementary and the  pre-requisites  are minimal:
Hilbert's result on the existence of polynomials over number fields
 with   $A_n$ as their  Galois groups, and a technical result 
on linear disjointness that can be found in the book by Jacobson.

This paper is organized as below. Section 2 lists the preliminaries,
and also lists major results of Perlis on clusters. In Section 3 
we provide  proofs of all our theorems. In Section 4, we reformulate the
concept and results in terms field extensions avoiding polynomials.
In Section 5, some  illustraive examples obtained by applying our
theorem are listed along with the SAGE code that produced them.
In the final  section we present further problems that need
investigation.

\vspace{3pt}
\noindent\textit{Acknowledgement}:\quad We would like to thank B.\ Sury 
for making us aware of the work of Perlis. We would also like to thank the anonymous referee who pointed out an error in an earlier
version of this paper.

\section{Preliminaries}
\label{sec:pre}
First we would like to summarize the results of Perlis that we feel need to 
be more widely known (and will also be used in our proofs).
For that purpose we set up the notation:

\begin{tabular}{rl}
	$K$ : & an infinite perfect field\\
	$f(x)\in K[x]$ :&  an irreducible polynomial of degree $n$\\
	$\alpha$ : & a  root of $f$ \\
	$K_f$ :& the splitting  field of $f$\\
	$G$ : & Gal$\,( K_f/ K)$\\
	$H\subset G$ :& the subgroup with  $K(\alpha)$ as the fixed field.
\end{tabular}

We define $r_K(f)$ as   the number of roots of $f$ lying in $K(\alpha)$, 
 and  $s_K(f)$ as the  number of distinct fields $K(\alpha ')$, with 
 $\alpha '$  varying over all the $n$ roots of $f(x)$.

 \noindent 
\begin{thm}(Perlis\cite{perlis}) \label{perthm}
 Under the notation above  we have:
	\begin{itemize}
\item [(i)] 
$r_K(f)$  is independent of the choice of the root $\alpha$ of $f(x)$.
\item [(ii)] $r_K(f)$ divides the degree of $f(x)$.\
\item [(iii)]In fact  $r_K(f) \cdot s_K(f) =\deg f$
\item [(iv)]$r_K(f)$ is the index of the subgroup $H$ in its 
	normaliser  $N_G(H)$.
\end{itemize}
\end{thm}  

The final part (iv) above makes it clear that
\begin{cor}\label{cseq}
	Given a finite separable extension of fields $L/K$ 
	the cluster size is the same for all irreducible polynomials
	 that are the minimal polynomials of primitive elements of
	 $L$ over $K$.
\end{cor}

\vspace{6pt}
First we list a group-theoretic result that is easy to prove:
\begin{prop}  \label{maximal}
For $n\geq 4$, any subgroup $H$  of index $n$ in the alternating group 
$A_n$   is  a  maximal subgroup and is isomorphic to $A_{n-1}$. 

\end{prop}
The next result assures us that linear disjointness is inherited
by all intermediate extension fields.
\begin{lemma}\label{jacobson}
	\cite[Lemma 1, Chapter 8.15]{jacob} Let $E$ be an extension of $F$. Let $E_1$ and $E_2$ be two extenions of $F$ contained in $E$, and $K_1$ an extension of $F$ contained in $E_1$. Then $E_1$ and $E_2$ are linearly disjoint over $F$ if and  only if the following two conditions hold:
\begin{itemize}
\item $K_1$ and $E_2$ are linearly disjoint over $F$ and 
\item $K_1E_2$ and $E_1$  are linearly disjoint over $K_1$.
\end{itemize}

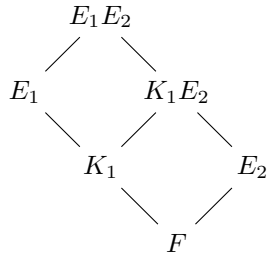
\begin{center}
\begin{tikzpicture}[node distance = 2cm, auto]
    \node (F) {$F$};
    \node (K1) [above of=F, left of=F, node distance=1cm] {$K_1$};
    \node (E1) [above of=K1,left of =K1, node distance=1cm] {$E_1$};
    \node (E2) [above of=F, right of=F,node distance=1cm] {$E_2$};
    \node (K1E2) [above of=F, node distance = 2cm] {$K_1E_2$};
    \node (E1E2) [above of=K1E2, left of =K1E2,node distance = 1cm] {$E_1E_2$};

    \draw (F)--(K1)  ;
    \draw (F)--(E2)  ;
    \draw (K1)--(K1E2);
    \draw (K1E2)--(E1E2);
    \draw (E2)--(K1E2);
    \draw (K1)--(E1);
    \draw (E1)--(E1E2);
\end{tikzpicture}
\end{center}
\end{lemma}
Next we present a simple result about number fields that is a crucial ingredient in 
proving Theorem \ref{application}:
\begin{lemma}\label{cyclotomic}
Given a finite extension of algebraic number field $K/F$ and a positive integer
$d\ge2$, there exist infintely many Galois extensions of $F$ 
of degree $d$ which are  linearly disjoint with $K$ over $F$.
\end{lemma}
\begin{proof}
	\textit{Case (i),} $F=\bq$:\qquad
	Let $\Delta_K\in\bz$ be the discriminant of the number field
$K$. For a prime number $p$ not dividing $\Delta_K$, consider the 
cyclotomic extension $L$ of $\bq$ obtained by adjoining a
primitive $p$-th root of unity. While 
$p$ is the only prime ramified in $L$  it remains
unramified in $K$,  which allows us to conclude that
	$L\cap K=\bq$. As $L$ is Galois over $\bq$
	it follows (see \cite{morandi}, page 184, Example 20.6) 
	now that $L$ and $K$ are linearly disjoint over $\bq$.
Now for a given integer  $d\ge2$  by Dirichlet's theorem on  
arithmetic progressions, we can find infinitely many primes $p$ such that
(i) $p>\Delta_K$ and  (ii) $p\equiv1\pmod d$.
With such choices of $p$ our cyclotomic fields $L$  will all have a (cyclic)
Galois extension of degree $d$ as a subfield  and they will again be
linearly disjoint with $K$ providing what we require.
Note that we get infinitely many such Galois extension.

	\textit{Case (ii),} Base field $F$ is any  number field:\qquad
	For given  extension $K/F$, temporarily forgetting $F$, by Case (i)
	we get extensions  $\bq(\beta)$ 
	over $\bq$ that are linearly disjoint with $K$ over
	$\bq$. Now by the Lemma~\ref{jacobson} above,
	the extension $F(\beta)$ over $F$ will be  linearly disjoint
	with $K$  over  the intermediate field $F$. 
	This linear disjointness property also guarantees 
	this extension $F(\beta)/F$
	will again be Galois and be of the same degree as $\bq(\beta)/\bq$.

	\end{proof}

\section{Proofs of  Main Results}

\subsection{Proof for Cluster Magnification}
\begin{proof}
Let the $n$ roots of $f(x)$ in an algebraic closure of $K$ be 
$\alpha=\alpha_1,\alpha_2,\ldots,\alpha_n$. Without loss of 
generality assume that the cluster of $\alpha$ comprises  the first
$r$ of these roots: 
$\{\alpha_1,\alpha_2,\ldots,\alpha_r\}$.
The hypothesis gives us some Galois extension 
$K(\beta)$ over $K$  that is linearly disjoint with $K_f$ over $K$.
 
Let the conjugates over $K$ of $\beta$ be $\beta=\beta_1,\beta_2,\ldots,\beta_d$. 
They lie in $K(\beta)$ (because it is a Galois extension
of $K$).

	Now $K$ being perfect the field extension $K(\alpha,\beta)$
	is a simple extension of $K$ generated by $\alpha+c\beta$ for a 
	suitable element $c\in K$. Using  $c\beta$ as  primitive element
	of $K(\beta)$ over $K$
	  we can assume $\alpha+\beta$ is
	a primitive element over $K$ for $K(\alpha,\beta)$.

We now claim that the minimal polynomial $g(x)$ of $\alpha+\beta$ 
which has degree $nd$,  has cluster size $rd$.

The conjugates over $K$ of the sum $\alpha+\beta$  are $\alpha_i+\beta_j$ 
for $i=1,2,\ldots,n$ and  $j=1,2,\ldots,d$.
As $\alpha_i$ for $1\le i\leq r$ are in $K(\alpha)$, we can see that
	$\alpha_i+\beta_j$ belong to $K(\alpha+\beta)=K(\alpha,\beta)$ 
	for $i=1,2,\ldots,r$
 and $j=1,2,\ldots,d$. So we get:

\begin{itemize}
	\item[(\textbf{CL})] The cluster of $\alpha+\beta$ has at least $rd$ roots out of a total of $nd$ roots of $g(x)$.
\end{itemize}
Now we want to show no more roots besides what is in assertion 
(\textbf{CL}) are in $K(\alpha+\beta)$. With that in mind 
we layout the fields in a diagram  with the
 compositum of $K_f$ and $K(\beta)$ at the top.

By the hypothesis of linear disjointness  of
$K(\beta)$ and $K_f$  over $K$, we get, using 
Lemma \ref{jacobson},  (see the diagram)
\begin{itemize}
	\item[(\textbf{LD1})]
	$K(\alpha)$ and $K(\beta)$ are linearly disjoint over $K$ and
\item[(\textbf{LD2})]
	$K_f$ and $K(\alpha+\beta)$ are also linearly disjoint
	over $K(\alpha)$.
	\end{itemize}
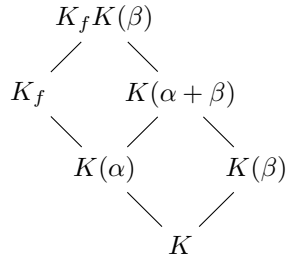
\label{lin_dis}
\begin{center}
\begin{tikzpicture}[node distance = 2cm, auto]
 \node (F) {$K$};
	\node (K1) [above of=F, left of=F, node distance=1cm] {$K(\alpha)$};
 \node (E1) [above of=K1,left of =K1, node distance=1cm] {$K_f$};
 \node (E2) [above of=F, right of=F,node distance=1cm] {$K(\beta)$};
	\node (K1E2) [above of=F, node distance = 2cm] {$K(\alpha+\beta)$};
\node (E1E2) [above of=K1E2, left of =K1E2,node distance = 1cm]{$K_fK(\beta)$};

    \draw (F)--(K1)  ;
    \draw (F)--(E2)  ;
    \draw (K1)--(K1E2);
    \draw (K1E2)--(E1E2);
    \draw (E2)--(K1E2);
    \draw (K1)--(E1);
    \draw (E1)--(E1E2);
\end{tikzpicture}
\end{center}

Suppose now, to the contrary,
the cluster of $\alpha+\beta$ has  more roots besides
the one shown in (\textbf{CL}).
Without loss of generality,  we may assume
 $\alpha_{r+1}+\beta_j \in K(\alpha+\beta)$ for some $j$.
As $ \beta_j\in K(\beta)\subset K(\alpha+\beta)$,
 we get  $\alpha_{r+1}\in K(\alpha+\beta)$. Being the
splitting field for $f(x)$, we know that  $K_f$ contains $\alpha_{r+1}$.
So we can see that  $\alpha_{r+1} \in K(\alpha+\beta)\cap K_f$.
From the linear disjointness  statement (\textbf{LD2}) we can also see that  
this intersection 
is $K(\alpha)$, and so  
$\alpha_{r+1}\in K(\alpha)$.

This is a contradiction to the fact that the root cluster of $\alpha$ is
$\{\alpha_1,\alpha_2,\ldots,\alpha_r\}$.
Thus the cluster size of $g(x)$ is $rd$.

\end{proof}

\subsection{Proof of Theorem \ref{application}}
In this subsection we take the base field to be any
algebraic number field.
First let us establish the following easy case  
which is perhaps known to experts in some other guise.
\begin{prop} \label{singleton}
For all integers $n\ge3$ there exist irreducible polynomials
$f(x)$ of degree $n$ in $K[x]$   for which all the root clusters are singletons.
\end{prop}

\begin{proof}
	By the well-known result of Hilbert,\cite{serre}
	for any given number
	field $K$ there exists an  irreducible 
polynomial $f(x)$ of degree $n$ over $K$ such that its splitting field $K_f$
	has  Galois group  $A_n$. (For the important
	special case of $K=\bq$, 
	the  Laguerre polynomials have this property.
	See \cite{altpoly}).

Let $\alpha$ be one of the roots  of  such a polynomial $f(x)$.
Then  $K(\alpha)$ is the fixed field of some subgroup
$H \subset A_n$ of index $n$.  

 By Proposition \ref{maximal}, $H$ is a proper maximal subgroup 
 of the simple group $A_n$, and  hence
	$N_G(H)=H$, the other possibility $N_G(H)=G$ is ruled out
	owing to the simplicty of the groups 
	$G=A_n$. Then by part (iv) of Theorem  \ref{perthm} we can conclude
 that the clusters are singletons.
\end{proof}

Now we have all  what is needed in place to  prove  Theorem \ref{application}.
This will be achieved by  appealing  to our theorem on Cluster Magnification.
For that we need  an input polynomial to start with.

We first  prove part (a) of our theorem.
The polynomial  $f(x)$  provided by the just-proved 
Proposition~\ref{singleton} fits 
our purpose. This has cluster size $r=1$. 
Let $K_f$ be its splitting field.
Now using Lemma~\ref{cyclotomic} we can find a Galois
extension $K(\beta)$ of degree $d$ linearly disjoint with $K_f$.
So Cluster Magification (Theorem~\ref{algorithm})
is applicable and we can conclude
that the minimal polynomial of $\alpha+\beta$ where
$\alpha$ is any root of $f(x)$ indeed has  cluster size $d$.

Proof of part (b): if we have  a degree 4 irreducible polynomial
with cluster size 2 we can then apply Magnification getting polynomials of degree $2d$ with cluster size $d$.
 We claim that quartic polynomials with Galois groups $D_4$ will always
have cluster size 2. The degree 4 extension given by such a quartic, 
as a subfield of the splitting field of degree 8,
has to be the fixed field of a subgroup $H$  generated by a reflection. 
Only the rotation of order 2 in $D_4$, being a central element, will be in the normalizer
$N(H)$ (besides  the elements of $H$) 
and so the index will be 2. And this index is the cluster size
as per part (iv) of  Theorem~\ref{perthm}.

For the case of
$\bq$, and $D_4$, Example 2 given in the Introduction 
provides polynomials with $D_4$ as Galois groups completing the
proof of (b).

\noindent\textbf{Remark}:\quad  
The arguments presented just now show that 
for any irreducible polynomial of degree $n$
whose Galois group is the dihedral group $D_n$, 
the cluster size is either 1 or 2 according as $n$ is odd or even.

\vspace{3pt}
\noindent\textbf{Caveat}: Our procedure does not yield irreducible polynomials
of degree $2n$  having root clusters of size $n$ when $n$ is odd.
\subsection{Proof of Theorem \ref{fixkroots}}
As in the hypothesis, $f(x)\in K[x]$ is an irreducible polynomial of degree $n$
with $S_n$ as Galois group. Let the roots be $\alpha_1,\alpha_2,\ldots, \alpha_n$. For $k<n$, define $L_k= K(\alpha_1,\alpha_2,\ldots,\alpha_k)$.
Let $H_k\subset S_n$ be the subgroup fixing $L_k$ elementwise. Clearly $H_k$ consists
of all permutations of the remaining $n-k$ roots and so isomorphic to $S_{n-k}$.
To compute the cluster size we use Theorem \ref{perthm} (iv). 
This involves computing the normalizer of $H_k$ in $S_n$ and the index
of $H_k$ in it.
For this we state the following lemma in  general form which is easy to prove
and hence omitted.
\begin{lemma}
	Let $\Omega$ be a finite set and $A$ a proper  nonempty subset.
	Let $S_\Omega$ be the group of all permutations of $\Omega$,
	with $S_A$  embedded naturally in $S_\Omega$.
	Then the normalizer of $S_A$ in $S_\Omega$ is $S_A\times S_{A^c}$, the
subgroup of permutations of $\Omega$ that preserve the subset $A$ (and hence its
	complement $A^c$).
	The index of $S_A$ in this normalizer is the order of $S_{A^c}$.
\end{lemma}
Now proof of our theorem follows immediately from the above lemma. 
\section{Reformulation}
The interpretation of root cluster size given 
in  part (iv) of Theorem \ref{perthm} makes it possible 
to carry out the whole discussion without mentioning irreducible
polynomials or their roots. We can go back to the general case of  perfect 
field  $K$ instead of 
$\bq$. We indicate below how to do this.

All the fields considered will be subfields of $\bar K$ and finite 
extensions of $K$.

Let us take such a field $L$ of degree $n$ over $K$ 
and let $\widetilde L$ be its  Galois closure with 
$G=\gal(\widetilde L/K)$. Let
$H=\{g\in G\mid g(a) = a\mbox{ for }a\in L\}$.

The index of $H$ in $N_G(H)$, denote it by $r$,  is now  the cluster
size of $L/K$. (See Corollary \ref{cseq})

We can reformulate Cluster Magnifying Theorem as below:

\textit{Given an extension $L/K$ of degree n with cluster size $r|n$, 
and any finite Galois extension F/K of degree $d$ which is
linearly disjoint  with $\widetilde{L}$ over K,  the compositum  $ LF$ 
having degree $nd$ over  $K$ will have cluster size $rd$ over $K$.}

To see the validity  in this formulation first  let us denote $\gal(F/K)$ 
by $\Gamma$. The linear disjointness ensures the Galois group of compositum
is the direct product.
So our Cluster Magnifying Theorem becomes index magnifying theorem for subgroups
under direct product.
That is a consequence of the naturality
statement below $$N_{G\times \Gamma}(H\times \{e\}) = N_G(H)\times \Gamma$$
proof of which  is a simple exercise in group theory.
\section{Examples}
Applying our cluster magnification process, starting with some easy
 polynomials with cluster size readily  computable as input, we computed 
higher degree polynomials with proportionately higher  cluster size.
The SAGE code  and the output for 3 cases are given below:
\begin{verbatim}
R.<x> = PolynomialRing(QQ)
f = x^5-3x-3 # cluster size is 1
L.<alpha,beta> = NumberField([f, x^4+x^3+x^2+x+1])
gamma = alpha+beta
g = gamma.absolute_minpoly()
print (g)
print ("Magnified Cluster size =", len(g.roots(ring=L)))
\end{verbatim}

\def\vghost{\vrule height14pt width0pt}
\[
\begin{array}{|l|l|l|l|l|}
	\hline
	\mbox{ Input: min poly of } \alpha   &\vghost \mbox{ Magnifier: min poly of
	}\beta & 
	  \mbox{ Magnified Output: min poly of } (\alpha+\beta)  \\ \hline 
x^4-2x^3+x-1  & x^3-3x-1 &\vghost  x^{12} - 6 x^{11} + 51 x^{9} - 51 x^{8} - 120 x^{7} 
	  \\\mbox{degree: } n=4 &\mbox{magnifying factor:}&{}+117 x^{6} + 144 x^{5} - 114 x^{4}- 33 x^{3} \\
	{\mbox{cluster size: } r =2} & d = 3 & 
	  {}+ 63 x^{2} - 72 x + 19 \\
	&&	\vghost \mbox{cluster size: } rd=6 \\ \hline
\hphantom{xx}
	x^3-2 &\vghost  x^3-3x-1 & x^{9} - 9 x^{7} - 9 x^{6} + 27 x^{5} - 54 x^{3} \\
	\mbox{degree: }n = 3&\mbox{magnifying factor:}& {}- 81 x^{2} + 81 x + 27\\
	{\mbox{cluster size: } r =1} & d=3 & 
	{\mbox{cluster size: } rd =3}
\\ \hline
x^5-3x-3  & x^4+ x^3+x^2 + x+1 &\vghost 
x^{20} + 5 x^{19} + 15 x^{18} + 35 x^{17} + 58 x^{16} \\
	\mbox{degree: } n = 5	&\mbox{magnifying factor:}&{}+61 x^{15} + 50 x^{14} + 165 x^{13}
	+ 944 x^{12} 
	\\
	\mbox{cluster size: }r =1 & d=4 & 
    {}+	2921 x^{11} + 5061 x^{10} + 5255 x^{9}  \\ &&{}+ 2882 x^{8} + 1092 x^{7} +3074 x^{6} \\
	&&
	{}+ 4681 x^{5} + 3481 x^{4} + 143 x^{3} \\ &&{} + 54 x^{2} + 602 x + 781\\
	&&\vghost \mbox{cluster size: } rd=4 \\ \hline
\end{array}
\]

\section{Further Problems}
There are many avenues open for further research on the concept
of root clusters.
A few  natural questions that arise in this context 
are listed  below:
\begin{enumerate}
	\item 
 Given an irreducible  polynomial of cluster size
${}>1$, what are the conditions that guarantee  it is obtained 
	by this magnification process from a 
lower degree polynomial with smaller  cluster size?
\item Produce examples of irreducible polynomials of higher cluster
	sizes that are not obtained by magnification process. (`primitive
		cases').
\item Let $\beta_1,\beta_2,\ldots,\beta_s$ be a complete set of 
representatives of the root clusters of an irreducible polynomial.
Now consider the following tower of fields terminating at the splitting field:
	\[ K\subset K(\beta_1)\subset K(\beta_1,\beta_2)\subset
	\cdots \subset K(\beta_1,\beta_2,\ldots,\beta_s)\]
What are the degree of these fields? Is this degree sequence independent of 
the ordering of the $\beta_i$'s? If the Galois group of $f(x)$ is $S_n$ the 
answer is easy, the degrees are: $n, n(n-1), n(n-1)(n-2),\ldots, n!$.
\item Relationship between the number of clusters  and the length of the above tower
seems to be mysterious. One may be tempted to believe that tower length
will be approxitely the number of distinct clusters.
Unfortunately it is not so.
For example there are irreducible polynomials of degree a prime number $p$
		with the semidirect product $C_p\rtimes C_{p-1}$ (a solvable group) 
as Galois group. Cluster size here  is not $p$ and so has to be 1.
 However a 
theorem of Galois on prime degree polynomials says the splitting field is 
generated by two roots when the Galois group is solvable. 
(See \S68 in \cite{edwards}). This shows that  
the tower length is always 3 though the number of clusters is $p$. 
\end{enumerate}

\end{document}